\documentclass[a4paper,10pt,intlimits,ngerman]{amsart}
\usepackage[varg]{txfonts}
\usepackage[unicode]{hyperref}
\usepackage{graphicx}
\usepackage{color}


\newtheorem{theorem}{Theorem}[section]

\newtheorem{lemma}[theorem]{Lemma}
\newtheorem{corollary}[theorem]{Corollary}

\newtheorem{remark}[theorem]{Remark}

\hypersetup{
breaklinks=true,
colorlinks=true,
linkcolor=blue,
citecolor=blue,
urlcolor=blue,
}

\numberwithin{equation}{section}

\DeclareMathOperator{\diam}{diam}

\DeclareMathOperator{\willmore}{W}
\DeclareMathOperator{\en}{W}
\DeclareMathOperator{\wen}{W}
\DeclareMathOperator{\ar}{\mu}
\DeclareMathOperator{\area}{area}
\DeclareMathOperator{\vol}{vol}

\DeclareMathOperator{\graph}{graph}

\newcommand{\grad}{\nabla_{L^2} \en}

\title[]{A Reverse Isoperimetric Inequality and its Application to the Gradient Flow of the Helfrich Functional}
\author{
Simon Blatt}
\address[Simon Blatt]{Paris Lodron Universit\"at Salzburg, Hellbrunner Strasse 34, 5020 Salzburg, Austria}
\email{simon.blatt@sbg.ac.at}

\begin{document}
\date{\today}
\maketitle 

\date{\today}

\begin{abstract}
We prove a quantitative reverse isoperimetric inequality for embedded surfaces with Willmore energy bounded away from $8\pi$.  We use this result to analyze the negative $L^2$ gradient flow of the Willmore energy plus a positive multiple of the inclosed volume. We show that initial surfaces of Willmore energy less than $8\pi$ with positive inclosed volume converge to a round point in finite or infinite time. 
 \end{abstract}
 
 \tableofcontents
 
 \section{Introduction}
  
The isoperimetric inequality is certainly one of the most famous and oldest inequalities in mathematics with many application in geometry and analysis. Though the origins of the problem go as far back as the founding myth on the ancient town of Carthago it has never stopped to inspire mathematics. 
An excellent account of the importance and  the rich and long history of this inequality can be found in the paper of Osserman  \cite{Osserman1978}. 

In $\mathbb R^3$ the isoperimetric inequality can be put as
\begin{equation}\label{eq:Isoperimetric}
 \vol(\Omega) \leq \frac {1}{6 \pi^{\frac 1 2}} (\ar(\partial \Omega))^{\frac 32}
\end{equation}
for all nice enough bounded domains $\Omega \subset \mathbb R^3$, where $\vol$ denotes the volume of $\Omega$ and $\ar(\Omega)=\area(\partial \Omega)$ the surface area of $\partial \Omega.$

Obviously the reverse of the isoperimetric inequality does not hold without extra assumptions. One can just think of a balloon and imagine what happens if one sticks a needle into it. Then of course the volume goes to zero as the air leaves the balloon but the area does not. In this note we will derive a condition that implies a reverse isoperimetric inequality which we stumbled over while studying the negative gradient flow of a special case of \emph{Helfrich's functional}. In its formulation we use the \emph{Willmore} or \emph{elastic energy}
$$
 \wen(f):=\int_{\Sigma} \|H_f\|^2 d \mu_f
$$
of an immersion $f: \Sigma \rightarrow \mathbb R^3$ with mean curvature $H_f = \frac 1 2 ( \kappa_1 + \kappa_2)$ and induced surface measure $\ar_f$. Willmore showed that the absolute minimum of the Willmore energy among closed immersed surfaces is $4\pi$ and is attained only by round spheres \cite{Willmore1967}.  

For an immersion $f:\Sigma \rightarrow \mathbb R^3$ we denote by $\vol(f)$ the signed inclosed volume and by $\mu(f)$ the surface area. 
We will see that an embedding $f : \Sigma \rightarrow \mathbb R^3$ of Willmore energy less than $8\pi$ satisfies the following inequality.
 
 \begin{theorem}[Reverse isoperimetric inequality] \label{thm:RevIsoperimetric}
  There is a constant $C< \infty$ such that
 \begin{equation} \label{eq:InvIsoperimetric}
  \ar(f)^{\frac 1 2} \leq \frac C {(8\pi - \wen(f))^3} |\vol(f)|^{\frac 13} 
 \end{equation}
 for all embeddings $f: \Sigma \rightarrow \mathbb R^3$ with $\willmore(f)  < 8\pi$ of smooth surfaces $\Sigma$ without boundary. 
 \end{theorem}
 
Note that it was observed by Li and Yau \cite{Li1982} that all immersed surfaces without boundary and Willmore energy strictly less than $8\pi$ are embedded. So the restriction to embeddings is not essential in the theorem above. 
 
 \begin{remark}
 After we proved Theorem \ref{thm:RevIsoperimetric}, we noticed that a non-quantitative version of the isoperimetric inequality above for the special case of spheres can be deduced from the findings of Johannes Schygulla in \cite{Schygulla2012}. 
 Schygulla considered the isoperimetric ratio 
$$
 I (f) = \left( 6 \sqrt \pi \right) ^{\frac 1 3}\frac{|\vol(f)|^{\frac 1 3}}{\ar(f)^{\frac 1 2}}
$$
and showed that every isoperimetric ratio in the Interval $(0,1]$ can be realized by an embedded sphere with Willmore energy strictly less than $8 \pi.$ Combining this with an adaptation of the machinery developed by Leon Simon in \cite{Simon1993}, he proved that the Willmore energy can be minimized among all embedded spheres with fixed isoperimetric ration $I(f) \in (0,1]$ by a surface with Willmore energy strictly less than $8 \pi$.
Furthermore, he shows that these surfaces of minimal energy converge to a doubly covered sphere for $I(f) \downarrow 0$ - which of course has Willmore energy $8\pi$. But this implies that for $M <  8\pi$ there is a constant $C=C(M)< \infty$ such that 
$$
 \ar(f)^{\frac 12} \leq C |\vol(f)|^{\frac 13}
$$
for all embeddings $f: \mathbb S^2 \rightarrow \mathbb R^3$ with $\wen(f) \leq M.$
It also shows that no reverse isoperimetric inequality is satisfied by all embeddings $f:\mathbb S ^2 \rightarrow \mathbb R^3$ satisfying 
$$
 \wen(f)< 8\pi.
$$
\end{remark}

We do not expect that the dependencies in the above stated reverse isoperimetric inequality above are optimal not to speak of the constant that can in principle be calculated. It should be an interesting and challenging task to craft a more direct proof without relying on the approximate graphic decomposition of Leon Simon in \cite{Simon1993}. This might then also point to a sharper estimate.

We will apply Theorem \ref{thm:RevIsoperimetric} to analyze the asymptotic behavior of the negative $L^2$-gradient flow of the Willmore energy plus a positive multiple of the volume. 
We will see that we can use the scaling of the volume term to show that the flow must shrink to a circle in finite or infinite time.  
To be more precise, we consider the negative $L^2$-gradient flow of the energy
$$
 \en_{\lambda}(f) = \wen(f) + \lambda\vol(f).
$$
 
This leads to the evolution equation
\begin{equation} \label{eq:ConstrainedWillmore}
\partial_t f = \grad_{\lambda}(f)
\end{equation}
 where
$$
 \grad_{\lambda}(f) = \Delta_f H_f + 2 H_f (H_f^2 - K_f)- \lambda \nu.
$$
(cf. \cite{Helfrich1973}).
Note that the functional $\en_{\lambda}$ is  a special case of the \emph{Helfrich functional} 
$$
 \en^{H_0}_{\lambda_1, \lambda_2}(f) := \int_{\Sigma} |H-H_0|^2 d\mu_f + \lambda_1 \ar(f) + \lambda_2 \vol(f) 
$$
with zero spontaneous curvature $H_0$. In \cite{Blatt2018} under the general assumption $\lambda_1 >0$ we gave conditions under which the negative gradient flow of $\en^0_{\lambda_1, \lambda_2}$ converges to a round point in finite time. But we could not deal with the case $\lambda_1 =0$ which is the content of the present article.

The following theorem in a sense closes the analysis of the negative gradient flow of the Helfrich functional we started in \cite{Blatt2018}.

 \begin{theorem} \label{thm:Asymptotics}
  Let $f:S^2 \times [0,T)$ be a maximal solution of the flow  \eqref{eq:ConstrainedWillmore} with $\willmore(f_0) < 8 \pi$ and $\vol(f_0) >0$. If we let
   $$
   c_t:= \frac 1 {\ar(f_t)} \int_{\Sigma} f_t(x) d \mu_{f_t}(x)
  $$
  denote the center of mass of the immersion $f_t$, then the rescaled flow
  $$
    \frac {2\pi^{\frac 1 2}} {\ar(f_t)^{\frac 1 2}} \left( f(\cdot, t) - c_t \right)
  $$
  converges to the unit sphere as $t \uparrow T$ as manifolds in the $C^\infty$ topology. 
\end{theorem}

It seems to be an interesting question whether in the situation of Theorem \ref{thm:Asymptotics} the singularities form in finite time or not. Unfortunately, in contrast to the cases discussed in \cite{Blatt2018}, our analysis is not able to answer this question. This is even more surprising keeping in mind the two facts that the rescaled flow converges to the unit sphere and that spheres converge to points in finite time under our evolution equation. Unfortunately, if one scales back the rescaled flow to the original one, one loses control over the gradient of the Willmore energy quite completely.


 \section{Proof of the Reverse Isoperimetric Inequality}

 \subsection{Preliminaries}

Let us gather some results that essentially go back to Leon Simon's paper \cite{Simon1993}. First of all we will use that there is a universal constant $C<\infty$ such that 
the density ratio satisfies
\begin{equation} \label{eq:DensityRatio}
 \frac {\ar(f(\Sigma) \cap B_r(x))}{r^2} \leq C \wen(f)
\end{equation}
(cf. \cite[Equation (A.16)]{Kuwert2004}).
The reverse isoperimetric inequality furthermore relies on two well-known facts related to the Willmore energy.  On the one hand, we will use the  following consequence of the monotonicity formula due to Leon Simon which can be seen as an extension of the inequality of Li and Yau \cite{Li1982}. We will use the short hand $B_\rho = B_\rho(0)$ for the ball of radius $\rho >0$ around the origin.

 \begin{lemma}[A Li-Yau type inequality \protect{\cite[Lemma 1.4]{Simon1993}}]\label{lem:LiYau}
  Suppose $M\subset \mathbb R^3$ is a compact surface without boundary, $ \partial B_\rho$ intersects $M$ transversely, and $M \cap B_\rho$ contains disjoint subsets $M_1, M_2$ with $\partial M_j \subset \partial B_{\rho}$, $M_j \cap B_{\theta\rho} \not= \emptyset$, and $|\partial M_j| \leq \beta \rho$, $j=1,2$, where $\theta \in (0, \frac 1 2)$ and $\beta >0$. 
  Then 
  $$
   \willmore(M) > 8 \pi - C \beta \theta
  $$
  where $C$ does not depend on $M, \beta, \theta.$
 \end{lemma}
 
 The second main ingredient to the proof of the reverse isoperimetric inequality is Leon Simon's approximate graphical decomposition lemma for surfaces with small Willmore energy.  
 For the convenience of the reader we restate this result in the version we are going to use later on. Roughly speaking this lemma tells us that if the Willmore energy of an embedded surface in a ball is small it consists mainly of graphs of Lipschitz functions with small Lipschitz constant. More precisely one has:
 
 \begin{lemma}[\protect{\cite[Lemma 2.1]{Simon1993}}] \label{lem:LocallyLipschitzGraphs}
 For any $\beta >0$ there is an $\varepsilon_0 = \varepsilon_0(n,\beta)$ such that if $M \subset \mathbb R^n$ is a compact surface such that for an $\varepsilon \in (0, \varepsilon_0]$ and a ball $B_{\rho } = B_\rho(0)$ we have
 $\partial M \cap \overline B _{\rho} = \emptyset$, $0 \in M$, $|M \cap \overline B_{\rho} | \leq \beta \rho ^2$, and $\int_{M \cap B_{\rho}} |A|^2 \leq \varepsilon \rho$, then the following holds:

 There are pairwise disjoint closed sets $P_1, \ldots , P_N \subset M$ with 
 $$
   \sum_{j=1}^N \diam P_j \leq C \varepsilon^{\frac 1 2} \rho
 $$
 and
 $$
  M \cap B_{\rho / 2 } \setminus \left( \cup _{j=1}^N P_j \right) = \left( \cup _ {i=0}^M \graph u_i \right) \cap B_{\rho/2}
 $$ 
 where each $u_i \in C^\infty(\overline \Omega_i, L_i^\bot)$ where $L_i$ is a plane in $\mathbb R^n$, $\Omega_i$ a smooth bounded connected domain in $L_i$ of the form $\Omega_i = \Omega_i ^0 \setminus \left( \cup_k d_{i,k} \right)$, where $\Omega_i^0$ is simply connected and $d_{i,k}$ are pairwise disjoint closed discs in $L_i$ which do not intersect $\partial \Omega_i^0$ with graph $u_i$ connected, and with
 $$
  N \leq C \beta \quad \text{and} \quad \rho^{-1} \sup_{\Omega_i} |u_i| + \sup_{\Omega_i} |Du_i| \leq C \varepsilon^{\frac 1 {2(2n-3)}}.
 $$
 Furthermore, for any $\sigma \in [\frac \rho 4, \frac \rho 2]$ with $\partial B_\sigma \cap \left( \bigcup_j P_j\right) = \emptyset$, we have
 $$
  M \cap \overline{B_\sigma} = \bigcup D_{\sigma,i}
 $$
 where each $D_{\sigma_i}$ is a topological disc with $graph(u_i) \cap \overline{B_\sigma} \subset D_{\sigma_i}$ and
 $$
  D_{\sigma_i} \setminus \graph u_i
 $$
 is the union of a subcollection of the $P_j$ and each $P_j$ is topologically a disc.
 \end{lemma}
 
\subsection{The Proof of Theorem \ref{thm:RevIsoperimetric}}

The idea behind the proof of Theorem \ref{thm:RevIsoperimetric} is first to pick a ball in which in conclusion of the approximate graphical decomposition lemma of Leon Simon, Lemma \ref{lem:LocallyLipschitzGraphs} holds. A simple covering argument will help us to pick such a ball centered on the manifold whose radius is comparable to the size of the manifold.
The approximate graphical decomposition lemma then tells us that inside this ball the manifold consists mainly of Lipschitz graphs and at least one of these graphs is getting close to the center of the manifold. If a second of these graphs would get close to the center we will get a contradiction using Lemma \ref{lem:LiYau} as the Willmore energy is known to be strictly smaller than $8\pi$. So all other parts of manifold have a certain distance to the center -- which shows that there must also be a certain amount of enclosed volume in our ball. 

\begin{proof}[Proof of Theorem \protect{\ref{thm:RevIsoperimetric}}]
Let $\varepsilon >0$ be so small that Lemma~\ref{lem:LocallyLipschitzGraphs} applies. For $r>0$, Besicovich's covering lemma yields that there is a collection of pairwise disjoint balls $B_r(x_i)$, $i=1, \ldots, M$ centered on $f(\mathbb S^2)$ such that 
\begin{equation*} \label{eq:estArea}
 \sum_{i=1}^M \ar(B_r(x_i) \cap f(\mathbb S^ 2)) \geq c \ar (f)
\end{equation*}
where $c$ is a universal constant.
As \eqref{eq:DensityRatio} together with $\wen(f)<8 \pi$ implies
$$
 \ar(B_r(x_i)\cap f(\mathbb S^2)) \leq C r^2
$$
for a constant $C < \infty$ we deduce that 
\begin{equation} \label{eq:estNumber}
 M \geq c r^{-2} \ar(f)
\end{equation}
for a suitably changed $c >0.$

Using that the balls $B_r(x_i)$ are pairwise disjoint and \eqref{eq:estNumber}, we deduce that there is a ball centered at an $x \in f(\Sigma)$ with radius $r$ such that
$$
 \wen(f^{-1}(B_{r}(x))) \leq \frac {\wen(f)} {c \ar (f) } r^2
$$
as otherwise
$$
 \wen(f) \geq \sum_{i=1}^M \wen(f^{-1}B_r(x_i)) > M \frac {\wen(f)} {c \ar (f) } r^2
 \geq \wen(f).
$$
We pick $r= \varepsilon \sqrt{\frac{ c \ar(f) } {\willmore(f)}}$ to get
\begin{equation} \label{eq:SmallWillmore}
 \wen(f^{-1}(B_{r}(x))) \leq \varepsilon^2.
\end{equation}

Let us for the following assume that $x=0$ and $\varepsilon < \varepsilon_0$ where $\varepsilon_0$ is the constant from the approximate graphical decomposition lemma, Lemma \ref{lem:LocallyLipschitzGraphs}.Then there is a radius $\rho \in [\frac r2, r]$ such that the assumptions of the approximate graphical decomposition lemma of Leon Simon are fulfilled and
$$
 |\Sigma \cap B_{\rho}(x)| \leq C \rho.
$$ 
Hence, Lemma \ref{lem:LiYau}  tells us that for $\theta= \frac  {8\pi - \wen(f)}{C }$ only one of the connected components $D_j$ of $\Sigma \cap B_{\rho}(0)$ can intersect $B_{\rho \theta}$ as otherwise
$$
\wen(f) \geq 8 \pi - C \theta.
$$ 

Let us assume that $D_1$ is this unique component intersecting $B_{\rho \theta}$. After some rotation of the surrounding space we can assume that $L_1$ is parallel to the $(x_1, x_2)$-plane.
As by assumption $0 \in \Sigma$, the component $D_1$ must contain $0$ and consists of parts of the graph of a Lipschitz function $u_1$ with Lipschitz constant less that $C \varepsilon^{\frac 16}$
up to some pimpels whose sum of  diameters can be estimated by $C \varepsilon^{\frac 1 2} \rho$. 
Hence, we have for all $x = (x_1, x_2, x_3)\in D_1 \cap B_{\theta \rho} $
$$
 |x_3| \leq C (\varepsilon^{\frac 16} \theta + \varepsilon^{\frac 1 2}) \rho.
$$
In other words
$$
 D_1 \subset \{|x_2|\leq C (\varepsilon^{\frac 1 6} \theta + \varepsilon^{\frac 1 2}) \rho\}.
$$ 
Let us now chose $\varepsilon = \min \{\varepsilon_0,(4C)^{-6},\theta^2 (4C)^{-2}\}$ so that $\varepsilon \geq c \theta^2$ for a suitable constant $c>0$.
Then the inclusion above implies that either $B_{\theta \rho} \cap \{x_n > \frac {\theta \rho} 2 \}$ or $B_{\theta \rho}  \cap \{x_n < \frac {\theta \rho} 2 \}$ is contained in the domain bounded by $f(\mathbb S ^2)$ and hence using $\rho \in [\frac r2,r]$,  $r= \varepsilon \sqrt{\frac {c \mu(f)}{\wen(f)}}$, $\varepsilon \geq c \theta^2$, and that $\theta = \frac  {8\pi - \wen(f)}{C }$ we get
$$
 |\vol(f)| \geq |B_{\theta \rho} \cap \{x_n > \frac {\theta \rho} 2 \}|
 \geq c (\theta \rho)^3  \geq c  \theta^3 r^3   
 \geq c (8\pi - \wen(f))^9  \ar(f)^{\frac 32}.
$$
This proves the reverse isoperimetric inequality \eqref{eq:InvIsoperimetric}.
\end{proof}

 

 \section{Application to the Constrained Willmore Flow}

\subsection{Vanishing of the Flow for \protect{$\mathbf{T=\infty}$}}

We will use the reverse isoperimetric inequality to prove that solutions to \eqref{eq:ConstrainedWillmore} with positive enclosed volume must vanish if they exist for all times in the sense that its enclosed volume, surface, and diameter goes to zero. In the proof we will use that Peter Topping \cite[Lemma 1]{Topping1998} showed that
\begin{equation} \label{eq:EstDiam}
\diam(f) \leq \frac {2}{\pi} \sqrt{\ar(f) \wen(f)}.
\end{equation}
This estimate is a sharpened version of \cite[Lemma 1.1]{Simon1993}.

\begin{theorem} \label{thm:vanishing}
 If the flow $f_t$ exists for all $t \in [0,\infty)$ and satisfies $\vol(f_t) >0$ for all $t\in [T,\infty)$, then we have $\vol(f_t) \rightarrow 0$, $\ar(f_t) \rightarrow 0$, and $\diam(f_t) \rightarrow 0$ as $t \rightarrow \infty.$ 
\end{theorem}

\begin{proof}
Let us start with an observation for a general immersion $f:\Sigma \rightarrow \mathbb R^3$. We consider the dilations 
$$ f_\alpha = \alpha (f-p)$$ around a fixed point $p$ in $f(\Sigma)$. Note that
the Willmore energy stays constant while the volume behaves like $\alpha^3 \vol(f_1)$. 
By the definition of the $L^2$ gradient we hence find
\begin{equation} \label{eq:Scaling}
 \int_{\Sigma} \grad_{\lambda} (f) (f-p) d\mu_f= 
\frac {d} {d\alpha} \left.\left(\en_{\lambda}(f_\alpha) \right) \right|_{\alpha=1} =   3 \lambda \vol(f).
\end{equation}
Using the Cauchy-Schwartz inequality together with \eqref{eq:EstDiam}, we get
\begin{multline}\label{eq:EstimateGradient}
  \int_{\Sigma} \grad_{\lambda} (f) (f-p) d\mu \leq \|\grad_{\lambda} (f)\|_{L^2(\mu)}\cdot \|f-p\|_{L^2(\mu)} \\ 
  \\ \leq \|\grad_{\lambda} (f)\|_{L^2(\mu)}  \diam(f) \sqrt{\ar(f)}  \\
  \leq \frac 2 {\pi} \sqrt{\en(f)}\|\grad_{\lambda} (f)\|_{L^2(\mu)} \mu(f).
\end{multline}
Combining equation \eqref{eq:Scaling} with the estimate \eqref{eq:EstimateGradient} we get from the reverse isoperimetric inequality \eqref{eq:InvIsoperimetric} 
\begin{equation} \label{eq:EstGrad}
 \|\grad_{\lambda} (f) \|^2_{L^2(\mu)} \geq  \frac {9 \pi^2 \lambda^2}{ 4 \willmore(f)} \frac{\vol(f)^2}{\ar(f)^2} \geq C \vol(f)^{\frac 23}.
\end{equation}

Let us now assume that the maximal time of existence satisfies $T = \infty.$  Then the inequality above shows that $\liminf_{t \rightarrow \infty }\vol(f_t) =0$ as otherwise the energy would become negative eventually. Applying the reverse isoperimetric inequality again, using \eqref{eq:EstDiam} and the fact that the Willmore energy is uniformly bounded along the flow,
we  get
$$
 \ar(f_t) \text{ and }\diam(f_t ) \rightarrow 0
$$
as $t \rightarrow \infty.$
\end{proof}

\subsection{Construction of a Blowup Limit}
 
For the convenience of the reader let us briefly repeat the essence of the blowup construction at a singularity in finite or infinite time \cite{McCoy2016} at the beginning of the proof of Theorem~1.4 on page 25. This result extends the construction of a blowup by Kuwert and Sch\"atzle for the Willmore flow \cite{Kuwert2001,Kuwert2004}. Exchanging the a-priori estimates for the Willmore flow by the a-priori estimates for the constrained Willmore flow of McCoy and Wheeler one finds that the following is true.

\begin{theorem} \label{thm:Blowup}
There are constants $\varepsilon_0 >0$ and $c_0 >0$ such that the following holds: Let $f: \Sigma \times [0,T) \rightarrow \mathbb R^3$ be the maximal solution to \eqref{eq:ConstrainedWillmore}. Then there is a sequence of times $t_j \uparrow T$, of radii $r_j > 0$ and points $x_j \in \mathbb R^n$
such that the rescaled flows
$$
 f_j: \Sigma \times [0,c_0] \rightarrow \mathbb R^3, f_j(p,t) := \frac 1 {r_j} (f(p,t_j + r_j^4 t)-x_j)
$$
satisfy
$$
\int_{f_j^{-1}(B_1(0))} \|A_{f_j}\|^2 d\mu_{f_j} \geq \varepsilon_0
$$
and converge smoothly locally to a smooth family of proper immersions 
$$
\tilde f : \tilde \Sigma \times [0,c_0] \rightarrow \mathbb R ^3
$$
in the following sense: We can represent 
$$
 f_j  (\phi_j,t) = \tilde f + u_j(\cdot,t)
$$
where
\begin{itemize}
 \item $\phi_j : \tilde f^{-1}(B_j(0)) \rightarrow U_j$ is a diffeomorphism,
 \item $f_j^{-1}(B_R) \subset U_j$ for $j\geq j(R)$,
 \item $u_j \in C^\infty(\tilde \Sigma\times [0,c_0], \mathbb R^n)$ is normal along $\tilde f$,
 \item $\|\nabla^k u_j \|_{L^\infty (\tilde f^{-1}(B_j(0)))} \rightarrow 0$ as $j \rightarrow 0$.
\end{itemize}
\end{theorem}

We will call such an immersion $\tilde f$ a blowup limit. Note that due to Theorem \ref{thm:vanishing} we know even in the case that $T=\infty$ the radii $r_j$ must converge to to zero. This will be crucial in our further analysis of the blowup limits.

\subsection{Convergence to a Round Point}
The next theorem shows that possible blowup limits are stationary and parametrize Willmore surfaces. It is an extension of Theorem 4.4 in \cite{McCoy2016}. McCoy and Wheeler have shown the result only under the assumption that the energy of the initial surface is close to the Willmore energy of a sphere $4\pi$.

\begin{theorem} \label{thm:BlowupWillmore}
Let $f: \Sigma \times [0,T) \rightarrow \mathbb R^3$ be the maximal solution to \eqref{eq:ConstrainedWillmore} with $\vol(f_t)>0$ for all $t\in [0,T)$.
Then the blowup limit $\tilde f: \tilde \Sigma_\infty \rightarrow \mathbb R^3$ constructed above does not depend on time and parametrizes a Willmore surface.
\end{theorem}

\begin{proof}
Using that $f$ satisfies equation \eqref{eq:ConstrainedWillmore} together with
$$
\Delta_{f_j} H_{f_j} +  = r_j^3 \Delta_f H_{f},
$$
$$
\nu_{f_j} = \nu_j,
$$
and
$$
\partial_t f_j = r_j^3 \partial_t f 
$$
we get from \eqref{eq:ConstrainedWillmore} that the $f_j$ satisfy
\begin{equation}
\partial_t f_j = \grad (f_j) +  r_j ^3 \lambda_2 \nu_{f_j}.
\end{equation}
Since $f_j$ converges to $\tilde f$ locally smoothly and $r_j \rightarrow 0$, this implies
\begin{equation}
\partial_t \tilde f = \grad (\tilde f).
\end{equation}
As
\begin{align*}
 \int_0^{c_0} &\left(\int_{\Sigma} \|\grad f_j(x,t) + \lambda_2 r_j \nu_{f_j} \|^2 d\mu_{f_j} \right) dt \\ &=  \int_{t_j}^{t_j+c_0r_j^4} \left(\int_{\Sigma} \|\grad_{ \lambda} f(x,t) \|^2 d\mu_{f_j} \right)dt \\
 &= \en_{\lambda}(t_j) -  \en_{\lambda}(t_j+c_0r_j^4) \\ & \rightarrow 0 
\end{align*}
and $r_j \rightarrow 0$ as $j \rightarrow \infty$, we deduce that $\grad \tilde f = 0.$
\end{proof}

Combining Theorem~\ref{thm:BlowupWillmore} with the classification of Willmore spheres due to Bryant \cite{Bryant1984} and the removability of point singularities of Kuwert and Sch\"atzle \cite{Kuwert2004} we get

\begin{corollary} \label{cor:RoundPoints}.
Let $f_t: \mathbb S^2 \rightarrow \mathbb R ^3$, $t \in [0,T)$  be a maximal solution to the constrained Willmore flow of spheres with positive enclosed volume and $\lim_{t\rightarrow T} \en(f_t) < 8 \pi $. Then the blowup limit constructed above is a round sphere.
\end{corollary}

\begin{proof}
 Let us first assume that $\hat \Sigma$ is compact. Since then the local convergence of the rescaled solution is in fact global, $\tilde \Sigma$ is a topological sphere. So $\tilde f $ is a Willmore sphere with energy below $8 \pi$ and thus is parametrizing a round sphere by the classification result of Bryant \cite{Bryant1984}.
 
 We now lead the case that $\tilde \Sigma$ is not compact to a contradiction as in \cite{Kuwert2004}. We can assume without loss of generality that $0 \notin \hat f (\hat \Sigma)$ since $\tilde f$ is proper. We consider the images of the $f_j$ under the inversion on the standard sphere $I:\mathbb R^3 \setminus \{0\} \rightarrow R^3 \setminus \{0\}, x \mapsto \frac x {|x|^2}$, which is well-defined for large enough $j \in \mathbb N$. The embeddings $\tilde f_j = I \circ f_j$ converge locally smoothly to the embedding $I \circ \tilde f $ in $\mathbb R^3 \setminus \{0\}$ and due to the M\"obius invariance of the Willmore energy
 $$
  \en( I \circ \tilde f ) \leq \liminf_{j \rightarrow \infty}  \en(\tilde f^j ) 
  =  \liminf_{j \rightarrow \infty}  \en( f^j ) < 8 \pi. 
 $$
 The M\"obius invariance of the Willmore energy also implies that $I\circ \tilde f$ is a Willmore surface away from $0$. 
 Due to the point removability result of Kuwert and Sch\"atzle, $\tilde f^\infty$ can be extended to a Willmore sphere. Hence, due to a result of Bryant, it must parametrize a round sphere. But this would imply that $\tilde f$ was a plane - which would contradict
 $$
  \int_{\tilde \Sigma}\|A_{\tilde f }\|^2 d\mu_{\tilde f } >0.
 $$
 Hence, $\tilde \Sigma$ must be compact which concludes the proof.
\end{proof}

\subsection{Proof of Theorem \ref{thm:Asymptotics}}

To complete the proof of Theorem \ref{thm:Asymptotics} i.e. to get from the subconvergence proven in Corollary \ref{cor:RoundPoints} to convergence, we will use the following a-priori estimates
proven by Kuwert and Sch\"atzle for the Willmore flow and by McCoy and Wheeler for the flow of the Helfrich functional.

\begin{theorem} [\protect{\cite[Theorem 3.1, Theorem 3.11]{McCoy2016}}] \label{thm:regularity}
Let $f: \Sigma \rightarrow \mathbb R ^3$ be a smooth immersion. There are absolute constants $\varepsilon_0 >0$ and $c_0 < \infty$ such that if $\rho >0$ is chosen with
$$
 \int_{f^{-1}(B_\rho (x))} |A_f|^2 d\mu_f \leq \varepsilon < \varepsilon_0
$$
for any $x \in \mathbb R ^3$, then the maximal time $T$ of smooth existence of the flow \eqref{eq:ConstrainedWillmore} satisfies
$$
 T \geq \frac 1 {c_0} \rho^4 := \delta
$$
and for $ t \in (0, \delta)$ we have
$$
 \|\nabla^k A\|_{L^\infty} \leq C_k \sqrt{\varepsilon} t^{-\frac {k+1} 4}
$$
where $C_k$ is a universal constant only depending on $k$.
\end{theorem}

Now we can finally prove Theorem \ref{thm:Asymptotics}

\begin{proof}[Proof of Theorem  \ref{thm:Asymptotics}]
Let us first note, that $\vol(f_0)>0$ together with the fact that $\en_\lambda(f_0) < 8 \pi$ implies that the enclosed volume stays positive along the flow. Otherwise there would be a first time $t_0$ for which we have $\vol(f_{t_0}) = 0 $. But then the immersion $f_{t_0}$ cannot be an embedding and hence due to the inequality of Li and Yau we must have $\en_{\lambda}(f_{t_0})= \wen(f_{t_0}) > 8 \pi$.

As the blowup limit must be a round sphere, we know that there is a sequence of times $t_j \uparrow T$, radii $0<r_j \downarrow 0$ and points $x_k$ such that
$$
 f_j:=\frac 1 {r_j}(f_{t_j} - x_j)
$$ 
converges to a round sphere smoothly as manifolds.
Hence,
$$
 \en_{\lambda}(f_j) = \wen(f_{t_j}) + \frac 1 {r_{j}^3} \vol(f_{t_j}) \rightarrow 4 \pi
$$
as $j \rightarrow 0$. As  $\vol(f_{t_j}) \geq 0$, $\wen(f_{t_j}) \geq 4\pi$ and $r_j \downarrow0$ this implies 
$$
 \lim_{t \rightarrow T }\en_{\lambda}(f_t) = \lim_{j\rightarrow \infty}\en_{\lambda} (f_{t_j}) = 4 \pi  
$$
Using again that $\vol(f_{t_j}) \geq 0$ and $\wen(f_{t_j}) \geq 4\pi$ we get
$$
 \wen(f_t) \rightarrow 4 \pi
$$
and
$$
 \vol(f_t) \rightarrow 0
$$
for $t \rightarrow \infty.$

We hence get using the scaling invariance of the Willmore energy that 
$$
 \int_{\mathbb S ^2} \|A^0_{\tilde f_t}\|^2 d\mu_f \rightarrow 0
$$ 
as $t\rightarrow T$.
 As $\ar(\tilde f_t)=4\pi$ we can apply the rigidity result of De Lellis and M\"uller \cite{DeLellis2005} to obtain
$$ 
 \|A- Id\|_{L^2(\mathbb S^2)} \leq C \|A^0_{f_t}\|_{L^2 (\mathbb S^2)}.
$$
But together with estimate of the density ratio \eqref{eq:DensityRation} this implies that there is a radius $r>0$ independent of $t$ such that 
$$
 \sup_{x \in \mathbb R^3} \int_{\tilde f_t^{-1}(B_r(x))} \|A_{\tilde f_t}\|^2 d\mu_{\tilde f_t} < \varepsilon
$$
for all $t \in [0,T)$.  Hence we can apply the a-priori estimate as stated in Theorem \ref{thm:regularity} to deduce that
\begin{equation} \label{eq:UniformEst}
 \|\nabla^k A\|_{L^\infty} \leq C_k
\end{equation}
for a constant $C_k < \infty$ not depending on time.

We will show that $\tilde f_t$ converges smoothly in the sense of manifolds to the unit sphere $\mathbb S^2$ by showing that every sequence of times $t_j \rightarrow T$
has a subsequence $t_{j_k}$ such that $f_{t_{j_j}}$ after converges smoothly to a unit sphere in the sense of manifolds. 

But this can be seen as follows. One first applies the uniform estimates \eqref{eq:UniformEst} to deduce that a subsequence converges to a compact  smooth embedded manifold as in the proof of Theorem 5.2 in \cite{Kuwert2004}. Using the different scalings of the terms building the gradient as in the proof of Theorem \ref{thm:BlowupWillmore} one then gets that $f:\Sigma \rightarrow \mathbb R^3$ is a Willmore sphere with energy less than $8\pi$ and hence a round sphere by the classification result of Bryant \cite{Bryant1984}.

\end{proof}

\begin{remark}
 Note that although we know that the rescaled flow converges to a round sphere, we cannot deduce from this fact that the volume is non-increasing close to a singularity simply from reversing the scaling. So a different technique is needed to show the vanishing of the flow in the case that the flow exists till $T=\infty$.
\end{remark}

\bibliographystyle{alpha}
\bibliography{../../Literaturdatenbank/BibTex/Master}

\begin{thebibliography}{DLM05}

\bibitem[Bla18]{Blatt2018}
Simon Blatt.
\newblock A note on singularities in finite time for the constrained willmore
  flow.
\newblock 2018.

\bibitem[Bry84]{Bryant1984}
Robert~L. Bryant.
\newblock A duality theorem for {W}illmore surfaces.
\newblock {\em J. Differential Geom.}, 20(1):23--53, 1984.

\bibitem[DLM05]{DeLellis2005}
Camillo De~Lellis and Stefan M\"{u}ller.
\newblock Optimal rigidity estimates for nearly umbilical surfaces.
\newblock {\em J. Differential Geom.}, 69(1):75--110, 2005.

\bibitem[Hel73]{Helfrich1973}
Wolfgang Helfrich.
\newblock Elastic properties of lipid bilayers: theory and possible
  experiments.
\newblock {\em Zeitschrift f{\"u}r Naturforschung C}, 28(11-12):693--703, 1973.

\bibitem[KS01]{Kuwert2001}
Ernst Kuwert and Reiner Sch{\"a}tzle.
\newblock The {W}illmore flow with small initial energy.
\newblock {\em J. Differential Geom.}, 57(3):409--441, 2001.

\bibitem[KS04]{Kuwert2004}
Ernst Kuwert and Reiner Sch{\"a}tzle.
\newblock Removability of point singularities of {W}illmore surfaces.
\newblock {\em Ann. of Math. (2)}, 160(1):315--357, 2004.

\bibitem[LY82]{Li1982}
Peter Li and Shing-Tung Yau.
\newblock A new conformal invariant and its applications to the willmore
  conjecture and the first eigenvalue of compact surfaces.
\newblock {\em Inventiones mathematicae}, 69(2):269--291, 1982.

\bibitem[MW16]{McCoy2016}
James McCoy and Glen Wheeler.
\newblock Finite time singularities for the locally constrained willmore flow
  of surfaces.
\newblock {\em Communications in Analysis and Geometry}, 24(4):843--886, 2016.

\bibitem[Oss78]{Osserman1978}
Robert Osserman.
\newblock The isoperimetric inequality.
\newblock {\em Bull. Amer. Math. Soc.}, 84(6):1182--1238, 1978.

\bibitem[Sch12]{Schygulla2012}
Johannes Schygulla.
\newblock Willmore minimizers with prescribed isoperimetric ratio.
\newblock {\em Arch. Ration. Mech. Anal.}, 203(3):901--941, 2012.

\bibitem[Sim93]{Simon1993}
Leon Simon.
\newblock Existence of surfaces minimizing the {W}illmore functional.
\newblock {\em Comm. Anal. Geom.}, 1(2):281--326, 1993.

\bibitem[Top98]{Topping1998}
Peter Topping.
\newblock Mean curvature flow and geometric inequalities.
\newblock {\em J. Reine Angew. Math.}, 503:47--61, 1998.

\bibitem[Wil67]{Willmore1967}
T.~J. Willmore.
\newblock Curvature of closed surfaces in {$E^{3}$}.
\newblock {\em Acta Ci. Compostelana}, 4:127--129, 1967.

\end{thebibliography}

\end{document}